\documentclass[12pt]{article}

\usepackage{amssymb}
\usepackage{amsthm}
\usepackage{fleqn}
\usepackage{graphicx}
\usepackage{epstopdf}
\usepackage{amsfonts}
\usepackage[numbers]{natbib}
\usepackage[colorlinks,citecolor=blue,urlcolor=blue]{hyperref}
\usepackage{amsmath}
\usepackage{geometry}
\usepackage{pifont}
\newtheorem{theorem}{Theorem}[section]
\newtheorem{lemma}[theorem]{Lemma}
\newtheorem{proposition}[theorem]{Proposition}

\newtheorem{remark}[theorem]{Remark}

\def\<{\left<}\def\>{\right>}\def\({\left(}\def\){\right)}

\def\<{\left<}\def\>{\right>}\def\({\left(}\def\){\right)}

\newfam\msbmfam\font\tenmsbm=msbm10\textfont
\msbmfam=\tenmsbm\font\sevenmsbm=msbm7
\scriptfont\msbmfam=\sevenmsbm

\def\<{\left<}\def\>{\right>}\def\({\left(}\def\){\right)}

\begin{document}

\title{Stochastic maximum principle for hybrid optimal control problems
under partial observation
\thanks{This work was supported by
the National Key R\&D Program of China (2022YFA1006102),
the National Natural Science Foundation of China (11801072, 11831010),
and the Fundamental Research Funds for the Central Universities (2242021R41175).}}

\author{Siyu Lv\thanks{School of Mathematics, Southeast University,
Nanjing 211189, China (lvsiyu@seu.edu.cn).}
\and
Jie Xiong\thanks{Department of Mathematics and SUSTech International Center for Mathematics,
Southern University of Science and Technology, Shenzhen 518055, China
(xiongj@sustech.edu.cn).}
\and
Wen Xu\thanks{Department of Mathematics,
Southern University of Science and Technology, Shenzhen 518055, China
(xuwenj9@gmail.com).}}

\date{}

\maketitle

\begin{abstract}
This paper is concerned with a partially observed hybrid optimal control problem,
where continuous dynamics and discrete events coexist and in particular,
the continuous dynamics can be observed while the discrete events,
described by a Markov chain, is not directly available.
Such kind of problem is first considered in the literature
and has wide applications in finance, management, engineering, and so on.
There are three major contributions made in this paper:
First, we develop a novel non-linear filtering method to convert
the partially observed problem into a completely observed one.
Our method relies on some delicate stochastic analysis technique related to
hybrid diffusions and is essentially different from the traditional filtering approaches.
Second, we establish a new maximum principle based on
the completely observed problem, whose two-dimensional state process consists of
the continuous dynamics and the optimal filter.
An important advantage of the maximum principle is that it takes a simple form
and is convenient to implement.
Finally, in order to illustrate the theoretical results,
we solve a linear quadratic (LQ) example using the derived maximum principle
to get some observable optimal controls.
\end{abstract}


\section{Introduction}

The hybrid diffusion is a two-component process $(X_{t},\alpha_{t})$ in which
the first component $X_{t}$ evolves according to a continuous diffusion process
whose drift and diffusion coefficients depend on the regime of $\alpha_{t}$,
where $\alpha_{t}$ is generally assumed to be a Markov chain.
As a result, the hybrid diffusion has the ability to capture more directly
the discrete events that are less frequent (occasional) but nevertheless
more significant to longer-term system behavior; for more details,
see Yin and Zhu \cite{YinZhu2010} and Yin and Zhang \cite{YinZhang2013}
and the references therein. In recent years, there is an increasing interest
in studying control problems of hybrid diffusions.
For example, stochastic maximum principles for hybrid optimal control problems
have been investigated by many researchers, including Donnelly \cite{Donnelly2011},
Zhang et al. \cite{ZES2012}, Li and Zheng \cite{LiZheng2015},
and Nguyen et al. \cite{NNY2020ESAIM,NYN2021AMO}.

Note that all the works aforementioned treated the problems with complete observation.
However, in practice one often encounters scenes where the Markov chain is not observable.
An immediate example is the stock market, which can be roughly divided as bull market
and bear market represented by a two-state Markov chain.
Typically, the price of a stock follows a geometric Brownian motion whose coefficients
switch between the two market trends, but the exact switching times
cannot be directly observed in the real marketplace; see, e.g.,
Rishel and Helmes \cite{RishelHelmes2006}, Dai et al. \cite{DaiZhangZhu2010},
and Xu and Yi \cite{XuYi2020}. Thus, it is very natural and appealing to study
the hybrid optimal control problems under partial observation, where only the
continuous dynamics $X_{t}$ can be observed while the Markov chain $\alpha_{t}$
cannot be observed. To our best knowledge, this kind of control problem is new
and has not been considered in the literature. Clearly, it has considerable impacts
in both theoretical analysis and practical applications, although with intrinsic
mathematical difficulties.

To deal with the problem, the key point is to convert the partially observed problem
into a completely observed one. In this connection, we first develop a novel non-linear
filtering method to estimate the current regime of the Markov chain $\alpha_{t}$
given the information $\mathcal{F}_{t}^{X}$ of the continuous dynamics $X_{t}$.
By combining the non-linear filtering theory for ordinary diffusions (Xiong \cite{Xiong2008}),
the Markov chain theory (Yao et al. \cite{YaoZhangZhou2006},
Yin and Zhu \cite{YinZhu2010}, and Yin and Zhang \cite{YinZhang2013}),
and some related stochastic analysis technique (Karatzas and Shreve \cite{KaratzasShreve1991}),
we obtain the stochastic differential equation (SDE) (\ref{general filtering equation})
satisfied by the filtering process (i.e., the filter equation) and a control problem
with complete observation. It is emphasized that our filtering method is essentially different
from and more general than the traditional approaches such as Wonham \cite{Wonham1965},
Liptser and Shiryayev \cite{LipsterShiryayev1977}, and Bj\"{o}rk \cite{Bjork1980}
in the sense that it gives the filtering equation (\ref{general filtering equation})
satisfied by the conditional expectation of a non-linear function,
not just the conditional probability, of the Markov chain.

For the completely observed control problem, the state pair is given by $(X_{t},\pi_{t})$
(see (\ref{CO state equation})), where $\pi_{t}$ is the so-called optimal filter.
It can be regarded as a classical optimal control problem with no Markov chain.
Thereby, existing results such as Peng \cite{Peng1993} apply to our problem
and then we establish the corresponding maximum principle.
Note that the maximum principle derived is simple and easy to carry out.
The result is new in the control theory. It also distinguishes itself from that
of Wang and Wu \cite{WangWu2009TAC}, Huang et al. \cite{HWX2009},
Wang and Yu \cite{WangYu2012}, and Wang et al. \cite{WWX2013SICON},
where maximum principles were established but under some different
partially observed contexts. As an illustrative example, a linear quadratic (LQ)
problem is solved using the maximum principle and some observable optimal controls
are obtained as a non-linear form in terms of the optimal filter and the adjoint process.

The rest of this paper is organized as follows.
Section \ref{section PO} formulates the control problem under partial observation.
Section \ref{section CO} converts the partially observed problem into a completely observed one.
Section \ref{section MP} establishes the maximum principle.
Section \ref{section LQ} focuses on the LQ case and obtains observable optimal controls.
Finally, Section \ref{section concluding remarks} concludes the paper with some further remarks.

\section{Problem formulation}\label{section PO}

Let $[0,T]$ be a finite time horizon. Let $(\Omega,\mathcal{F},P)$ be a fixed
probability space on which a one-dimensional standard Brownian motion $W_{t}$,
$t\in [0,T]$, and a Markov chain $\alpha_{t}$, $t\in [0,T]$, are defined.
Assume that $W$ and $\alpha$ are independent. For simplicity, we assume that
$\alpha$ is a two-state Markov chain taking values in $\mathcal{M}=\{1,2\}$
with generator
\begin{equation*}
%\left\{
\begin{aligned}
\left[
  \begin{array}{cc}
    -\lambda_{1} & \lambda_{1} \\
    \lambda_{2} & -\lambda_{2} \\
  \end{array}
\right],
\end{aligned}
%\right.
\end{equation*}
where $\lambda_{1}>0$ and $\lambda_{2}>0$. Let $\{\mathcal{F}_{t}\}_{t\in [0,T]}$
be the natural filtration generated by $W$ and $\alpha$.

Throughout the paper, we denote by $\langle\cdot,\cdot\rangle$ the scalar product
of a Euclidean space. $A^{\top}$ denotes the transpose of a vector or matrix $A$.
$\varphi_{x}$ denotes the derivative of a function $\varphi$ with respect to $x$.
Given a filtration $\{\mathcal{G}_{t}\}_{t\in [0,T]}$, if $\phi$ is a real-valued
$\mathcal{G}_{t}$-adapted square integrable process
(i.e., $E\int_{0}^{T}|\phi_{t}|^{2}dt<\infty$),
we write $\phi\in L_{\mathcal{G}}^{2}(0,T;R)$.

Let the state of the system be described by the following controlled SDE:
\begin{equation}\label{PO system}
\left\{
\begin{aligned}
dX_{t}=&b(t,X_{t},\alpha_{t},v_{t})dt+\sigma(t,X_{t},\alpha_{t},v_{t})dW_{t},\\
X_{0}=&x\in R,\quad \alpha_{0}=i\in\mathcal{M},
\end{aligned}
\right.
\end{equation}
where $b,\sigma:[0,T]\times R\times \mathcal{M}\times U\mapsto R$ are some
deterministic functions and the non-empty convex set $U\subset R$ is called
the control domain. Let $\{\mathcal{F}_{t}^{X}\}_{t\in [0,T]}$ be the natural
filtration generated by $X$, representing the information available;
note that the Markov chain $\alpha$ cannot be directly observed for the controller.
A process $v$ is called an admissible control if it satisfies
$v\in L_{\mathcal{F}^{X}}^{2}(0,T;U)$; we also denote the set of all admissible controls
by $\mathcal{U}_{ad}=L_{\mathcal{F}^{X}}^{2}(0,T;U)$.
Now we introduce the following assumption:

\textbf{(A1)} The functions $b$ and $\sigma$ are continuously differentiable
with respect to $(x,v)$, and the partial derivatives
$b_{x}$, $b_{v}$, $\sigma_{x}$, $\sigma_{v}$ are uniformly bounded.

Under Assumption (A1), for any $v\in\mathcal{U}_{ad}$, SDE (\ref{PO system})
admits a unique solution $X$ (see, e.g., Yin and Zhu \cite{YinZhu2010}).
The associated cost functional is defined as
\begin{equation}\label{PO cost functional}
%\left\{
\begin{aligned}
J(v)=E\bigg[\int_{0}^{T}f(t,X_{t},\alpha_{t},v_{t})dt+g(X_{T},\alpha_{T})\bigg],
\end{aligned}
%\right.
\end{equation}
where $f:[0,T]\times R\times\mathcal{M}\times U\mapsto R$
and $g:R\times\mathcal{M}\mapsto R$ are some deterministic functions
satisfying the following assumption:

\textbf{(A2)} The function $f$ is continuously differentiable with respect to $(x,v)$
and the function $g$ is continuously differentiable with respect to $x$.
Moreover, there exists a constant $K>0$ such that
\begin{equation*}
%\left\{
\begin{aligned}
(1+|x|^{2}+|v|^{2})^{-1}|f(t,x,i,v)|
+(1+|x|+|v|)^{-1}(|f_{x}(t,x,i,v)|+|f_{v}(t,x,i,v)|)&\leq K,\\
(1+|x|^{2})^{-1}|g(x,i)|+(1+|x|)^{-1}|g_{x}(x,i)|&\leq K.
\end{aligned}
%\right.
\end{equation*}
The partially observed hybrid optimal control problem
is to find a $u\in\mathcal{U}_{ad}$ such that
\begin{equation*}
%\left\{
\begin{aligned}
J(u)=\min_{v\in\mathcal{U}_{ad}}J(v)
\end{aligned}
%\right.
\end{equation*}
subject to (\ref{PO system}).
If such a $u$ exists, then it is called an optimal control, and the corresponding
solution $X$ to (\ref{PO system}) is called the optimal trajectory.
Our main goal is to establish a maximum principle, namely, a necessary condition
for the optimal control $u$.
\begin{remark}
In this paper, we have assumed that the state process $X$, control process $v$,
and Brownian motion $W$ to be one-dimensional only for convenience of presentation.
There is no essential difficulty to extend the results to the multi-dimensional case.
\end{remark}

\section{Problem with complete observation}\label{section CO}

Note that, in the problem, only the state process $X_{t}$ is observable at time $t$.
The Markov chain $\alpha_{t}$ can not be directly available. Hence, it is necessary
to convert the problem into a completely observable one. One way to accomplish this
purpose is to use the non-linear filtering theory. To this end, we need to assume:

\textbf{(A3)} The function $\sigma$ is independent of the Markov chain $\alpha$, i.e.,
\begin{equation*}
%\left\{
\begin{aligned}
\sigma(t,x,i,v)=\sigma(t,x,v).
\end{aligned}
%\right.
\end{equation*}
Otherwise, we will have a singular filtering problem
(see Crisan et al. \cite{CKX2009}).

Denote
\begin{equation*}
%\left\{
\begin{aligned}
h(t,x,i,v)=\sigma^{-1}(t,x,v)b(t,x,i,v).
\end{aligned}
%\right.
\end{equation*}
We also make the following assumption:

\textbf{(A4)} The function $\sigma^{-1}$ is bounded.

Define the filtering process associated with a function
$\varphi:\mathcal{M}\mapsto R$
\begin{equation*}
%\left\{
\begin{aligned}
\mu_{t}(\varphi)=E[\varphi(\alpha_{t})|\mathcal{F}_{t}^{X}].
\end{aligned}
%\right.
\end{equation*}
Define the observation process
\begin{equation*}
%\left\{
\begin{aligned}
dY_{t}
=&\sigma^{-1}(t,X_{t},v_{t})dX_{t}\\
=&h(t,X_{t},\alpha_{t},v_{t})dt+dW_{t}.
\end{aligned}
%\right.
\end{equation*}
It follows that
\begin{equation}\label{FY in FX}
%\left\{
\begin{aligned}
\mathcal{F}_{t}^{Y}\subset\mathcal{F}_{t}^{X},\quad t\in [0,T],
\end{aligned}
%\right.
\end{equation}
where $\{\mathcal{F}_{t}^{Y}\}_{t\in [0,T]}$ denotes the natural filtration
generated by $Y$.

Define the innovation process
\begin{equation*}
%\left\{
\begin{aligned}
d\nu_{t}
=&dY_{t}-\mu_{t}(h)dt\\
=&h(t,X_{t},\alpha_{t},v_{t})dt+dW_{t}-\mu_{t}(h)dt,
\end{aligned}
%\right.
\end{equation*}
where with a bit abuse of notation, we denote
$\mu_{t}(h)=E[h(t,X_{t},\alpha_{t},v_{t})|\mathcal{F}_{t}^{X}]$.
\begin{lemma}\label{lemma innovation process}
The process $\nu_{t}$ is an $\mathcal{F}_{t}^{X}$-Brownian motion under $P$.
\end{lemma}
\begin{proof}
For simplicity, we denote $h_{t}=h(t,X_{t},\alpha_{t},v_{t})$.
Note that for $t>s$,
\begin{equation*}
%\left\{
\begin{aligned}
E[\nu_{t}-\nu_{s}|\mathcal{F}_{s}^{X}]
=&E\bigg[W_{t}-W_{s}+\int_{s}^{t}(h_{r}-\mu_{r}(h))dr\bigg|\mathcal{F}_{s}^{X}\bigg]\\
=&0+\int_{s}^{t}E[h_{r}-\mu_{r}(h)|\mathcal{F}_{s}^{X}]dr=0.
\end{aligned}
%\right.
\end{equation*}
In addition, it is clear that the quadratic variation process
$\langle\nu\rangle_{t}=\langle W\rangle_{t}=t$.
From L\'{e}vy's martingale characterization of Brownian motion
(see Karatzas and Shreve \cite[Theorem 3.16]{KaratzasShreve1991}
or Xiong \cite[Theorem 3.3.13]{Xiong2008}), we see that
$\nu_{t}$ is an $\mathcal{F}_{t}^{X}$-Brownian motion under $P$.
\end{proof}
Define a process
\begin{equation*}
%\left\{
\begin{aligned}
M_{t}^{-1}=\text{exp}\bigg(-\int_{0}^{t}h_{s}dW_{s}
-\frac{1}{2}\int_{0}^{t}|h_{s}|^{2}ds\bigg).
\end{aligned}
%\right.
\end{equation*}
Let $\widehat{P}$ be the probability measure absolutely continuous
with respect to $P$ with Radon-Nikodym derivative $M_{t}^{-1}$, i.e.,
\begin{equation*}
%\left\{
\begin{aligned}
\frac{d\widehat{P}}{dP}\bigg|_{\mathcal{F}_{t}}=M_{t}^{-1}.
\end{aligned}
%\right.
\end{equation*}
From Yao et al. \cite[Lemma 1]{YaoZhangZhou2006}, under the new probability measure
$\widehat{P}$, $Y_{t}$ is a standard $\widehat{P}$-Brownian motion,
$\alpha_{t}$ is still a Markov chain with the same generator under $P$,
and $Y_{t}$ is independent of $\alpha_{t}$. Moreover,
\begin{equation*}
%\left\{
\begin{aligned}
dM_{t}=h_{t}M_{t}dY_{t}.
\end{aligned}
%\right.
\end{equation*}
It implies that $M_{t}$ is also independent of $\alpha_{t}$ under $\widehat{P}$.

Denote
\begin{equation}\label{indicator function}
%\left\{
\begin{aligned}
\pi_{t}
=\mu_{t}(1_{\{\alpha_{t}=1\}})
=E[1_{\{\alpha_{t}=1\}}|\mathcal{F}_{t}^{X}]
=P(\alpha_{t}=1|\mathcal{F}_{t}^{X}),
\end{aligned}
%\right.
\end{equation}
i.e., $\pi_{t}$ is the so-called optimal filter.
To derive the SDE satisfied by $\mu_{t}$ (and specially, $\pi_{t}$),
we need the following additional assumption:

\textbf{(A5)} $\mathcal{F}_{t}^{X}=\mathcal{F}_{t}^{Y}$, $t\in [0,T]$.

The following remark gives some sufficient conditions under which
Assumption (A5) holds.
\begin{remark}
The Assumption (A5) is satisfied under any one of the following conditions:

{\rm(a)} The function $\sigma$ does not depend on the control $v$,
namely, $\sigma(t,x,v)=\sigma(t,x)$.

{\rm(b)} The admissible control is restricted to those $v_{t}$
which is $\mathcal{F}_{t}^{Y}$-adapted.

{\rm(c)} The admissible control is restricted to those $v_{t}$ which is
in a closed-loop form $v_{t}=v(t,X_{t})$ satisfying the Lipschitz condition:
\begin{equation*}
%\left\{
\begin{aligned}
|v(t,x)-v(t,y)|\leq L|x-y|,
\end{aligned}
%\right.
\end{equation*}
for some constant $L>0$.
\end{remark}
\begin{proof}
(a) Note that in this case, the SDE
\begin{equation*}
%\left\{
\begin{aligned}
dX_{t}=\sigma(t,X_{t})dY_{t}
\end{aligned}
%\right.
\end{equation*}
admits a unique strong solution $X_{t}=F(t,Y)$ for a suitable functional $F$
depending on the path of $Y$ up to $t$.
So we have $\mathcal{F}_{t}^{X}\subset\mathcal{F}_{t}^{Y}$, which combines with
(\ref{FY in FX}) implying that $\mathcal{F}_{t}^{X}=\mathcal{F}_{t}^{Y}$.

(b) We can modify the proof of (a) to get $X_{t}=F(t,v,Y)$ for a suitable functional $F$
depending on the paths of $(v,Y)$ up to $t$. Since $v_{t}$ is $\mathcal{F}_{t}^{Y}$-adapted,
we also have $\mathcal{F}_{t}^{X}\subset\mathcal{F}_{t}^{Y}$.

(c) In this case, we have the following SDE
\begin{equation*}
%\left\{
\begin{aligned}
dX_{t}=\sigma(t,X_{t},v(t,X_{t}))dY_{t}.
\end{aligned}
%\right.
\end{equation*}
Suppose that $X$ and $\widetilde{X}$ are two solutions to the above equation,
then we have
\begin{equation*}
%\left\{
\begin{aligned}
\widehat{E}|X_{t}-\widetilde{X}_{t}|^2
=\widehat{E}\int_{0}^{t}|\sigma(s,X_{s},v(s,X_{s}))
-\sigma(s,\widetilde{X}_{s},v(s,\widetilde{X}_{s}))|^2ds
\leq K\int_{0}^{t}\widehat{E}|X_{s}-\widetilde{X}_{s}|^2ds,
\end{aligned}
%\right.
\end{equation*}
where $\widehat{E}$ denotes the expectation under $\widehat{P}$.
It follows from Gronwall's inequality that $X_{t}\equiv\widetilde{X}_{t}$.
This yields the existence of a unique strong solution $X$,
which leads to $\mathcal{F}_{t}^{X}\subset\mathcal{F}_{t}^{Y}$.
\end{proof}
%Now we present the main result of this section.
\begin{theorem}\label{theorem filter}
Let Assumptions (A1)-(A5) hold. Then, the two-dimensional process $(X_{t},\pi_{t})$
defined by (\ref{PO system}) and (\ref{indicator function}), respectively,
is the unique solution to the following SDE:
\begin{equation}\label{CO state equation}
\left\{
\begin{aligned}
dX_{t}=&[b(t,X_{t},1,v_{t})\pi_{t}+b(t,X_{t},2,v_{t})(1-\pi_{t})]dt+\sigma(t,X_{t},v_{t})d\nu_{t},\\
d\pi_{t}=&[-\lambda_{1}\pi_{t}+\lambda_{2}(1-\pi_{t})]dt
+[h(t,X_{t},1,v_{t})-h(t,X_{t},2,v_{t})]\pi_{t}(1-\pi_{t})d\nu_{t},\\
X_{0}=&x\in R,\quad \pi_{0}=\pi\in \{0,1\},
\end{aligned}
\right.
\end{equation}
and the cost functional (\ref{PO cost functional}) becomes
\begin{equation}\label{CO cost functional}
%\left\{
\begin{aligned}
J(v)
=&E\bigg[\int_{0}^{T}[f(t,X_{t},1,v_{t})\pi_{t}+f(t,X_{t},2,v_{t})(1-\pi_{t})]dt\\
&+g(X_{T},1)\pi_{T}+g(X_{T},2)(1-\pi_{T})\bigg].
\end{aligned}
%\right.
\end{equation}
\end{theorem}
\begin{proof}
By Bayes's formula, we have
\begin{equation*}
%\left\{
\begin{aligned}
E[\varphi(\alpha_{t})|\mathcal{F}_{t}^{X}]
=\frac{\widehat{E}[\varphi(\alpha_{t})M_{t}|\mathcal{F}_{t}^{X}]}{\widehat{E}[M_{t}|\mathcal{F}_{t}^{X}]}.
\end{aligned}
%\right.
\end{equation*}
The above equation can be rewritten as
\begin{equation}\label{K-S formula}
%\left\{
\begin{aligned}
\mu_{t}(\varphi)=\frac{V_{t}(\varphi)}{V_{t}(1)},
\end{aligned}
%\right.
\end{equation}
where we denote
\begin{equation*}
%\left\{
\begin{aligned}
V_{t}(\varphi)=\widehat{E}[\varphi(\alpha_{t})M_{t}|\mathcal{F}_{t}^{X}].
\end{aligned}
%\right.
\end{equation*}
In view of Yin and Zhang \cite[Theorem 2.5]{YinZhang2013},
the process
\begin{equation*}
%\left\{
\begin{aligned}
N_{t}(\varphi)=\varphi(\alpha_{t})
-\int_{0}^{t}Q\varphi(\cdot)(\alpha_{s})ds
\end{aligned}
%\right.
\end{equation*}
is a martingale, where
\begin{equation*}
%\left\{
\begin{aligned}
Q\varphi(\cdot)(i)=\sum_{j\neq i}q_{ij}[\varphi(j)-\varphi(i)],\quad i\in\mathcal{M},
\end{aligned}
%\right.
\end{equation*}
is the infinitesimal operator associated with the Markov chain $\alpha$.

Applying integration by parts formula
(see Karatzas and Shreve \cite[Problem 3.12]{KaratzasShreve1991}), we obtain
\begin{equation*}
%\left\{
\begin{aligned}
d(\varphi(\alpha_{t})M_{t})
=&\varphi(\alpha_{t})dM_{t}+M_{t}d\varphi(\alpha_{t})+d\langle M,\varphi(\alpha)\rangle_{t}\\
=&\varphi(\alpha_{t})h_{t}M_{t}dY_{t}
+M_{t}Q\varphi(\cdot)(\alpha_{t})dt+M_{t}dN_{t}(\varphi),
\end{aligned}
%\right.
\end{equation*}
where, from the independence of $M$ and $\alpha$,
the cross-variation process $\langle M,\varphi(\alpha)\rangle_{t}\equiv0$.

Integrating from 0 to $t$, we have
\begin{equation*}
%\left\{
\begin{aligned}
\varphi(\alpha_{t})M_{t}=\varphi(\alpha_{0})
+\int_{0}^{t}\varphi(\alpha_{s})h_{s}M_{s}dY_{s}
+\int_{0}^{t}M_{s}Q\varphi(\cdot)(\alpha_{s})ds
+\int_{0}^{t}M_{s}dN_{s}(\varphi).
\end{aligned}
%\right.
\end{equation*}
Taking conditional expectation $\widehat{E}[\cdot|\mathcal{F}_{t}^{X}]$
on both sides, noting that $\mathcal{F}_{t}^{X}=\mathcal{F}_{t}^{Y}$ and
$Y$ and $\alpha$ are independent under $\widehat{P}$,
and using Xiong \cite[Lemma 5.4]{Xiong2008}, we get
\begin{equation*}
%\left\{
\begin{aligned}
V_{t}(\varphi)
=&V_{0}(\varphi)
+\widehat{E}\bigg[\int_{0}^{t}\varphi(\alpha_{s})h_{s}M_{s}dY_{s}\bigg|\mathcal{F}_{t}^{X}\bigg]\\
&+\widehat{E}\bigg[\int_{0}^{t}M_{s}Q\varphi(\cdot)(\alpha_{s})ds\bigg|\mathcal{F}_{t}^{X}\bigg]
+\widehat{E}\bigg[\int_{0}^{t}M_{s}dN_{s}(\varphi)\bigg|\mathcal{F}_{t}^{X}\bigg]\\
=&V_{0}(\varphi)
+\int_{0}^{t}\widehat{E}[\varphi(\alpha_{s})h_{s}M_{s}|\mathcal{F}_{s}^{X}]dY_{s}
+\int_{0}^{t}\widehat{E}[M_{s}Q\varphi(\cdot)(\alpha_{s})|\mathcal{F}_{s}^{X}]ds\\
=&V_{0}(\varphi)
+\int_{0}^{t}V_{s}(h\varphi)dY_{s}
+\int_{0}^{t}V_{s}(Q\varphi)ds.
\end{aligned}
%\right.
\end{equation*}
Hence, $V_{t}$ satisfies the following equation:
\begin{equation*}
%\left\{
\begin{aligned}
dV_{t}(\varphi)=V_{t}(Q\varphi)dt+V_{t}(h\varphi)dY_{t}.
\end{aligned}
%\right.
\end{equation*}
In particular,
\begin{equation*}
%\left\{
\begin{aligned}
dV_{t}(1)=V_{t}(h)dY_{t},
\end{aligned}
%\right.
\end{equation*}
and
\begin{equation*}
%\left\{
\begin{aligned}
d\frac{1}{V_{t}(1)}
=-\frac{V_{t}(h)}{V_{t}^{2}(1)}dY_{t}+\frac{V_{t}^{2}(h)}{V_{t}^{3}(1)}dt.
\end{aligned}
%\right.
\end{equation*}
Applying It\^{o}'s formula to (\ref{K-S formula}), we have
\begin{equation*}
%\left\{
\begin{aligned}
d\mu_{t}(\varphi)
=&d\frac{V_{t}(\varphi)}{V_{t}(1)}\\
=&\frac{1}{V_{t}(1)}dV_{t}(\varphi)+V_{t}(\varphi)d\frac{1}{V_{t}(1)}+dV_{t}(\varphi)d\frac{1}{V_{t}(1)}\\
=&\frac{1}{V_{t}(1)}[V_{t}(Q\varphi)dt+V_{t}(h\varphi)dY_{t}]
+V_{t}(\varphi)\bigg[-\frac{V_{t}(h)}{V_{t}^{2}(1)}dY_{t}+\frac{V_{t}^{2}(h)}{V_{t}^{3}(1)}dt\bigg]
-\frac{V_{t}(h\varphi)V_{t}(h)}{V_{t}^{2}(1)}dt\\
=&\bigg[\frac{V_{t}(Q\varphi)}{V_{t}(1)}
+\frac{V_{t}(\varphi)V_{t}^{2}(h)}{V_{t}^{3}(1)}
-\frac{V_{t}(h\varphi)V_{t}(h)}{V_{t}^{2}(1)}\bigg]dt
+\bigg[\frac{V_{t}(h\varphi)}{V_{t}(1)}
-\frac{V_{t}(\varphi)V_{t}(h)}{V_{t}^{2}(1)}\bigg]dY_{t}\\
=&[\mu_{t}(Q\varphi)+\mu_{t}(\varphi)\mu_{t}^{2}(h)-\mu_{t}(h\varphi)\mu_{t}(h)]dt
+[\mu_{t}(h\varphi)-\mu_{t}(\varphi)\mu_{t}(h)]dY_{t}.
\end{aligned}
%\right.
\end{equation*}
Replacing $dY_{t}$ by $d\nu_{t}+\mu_{t}(h)dt$ in the above equation, it follows that
the filtering process $\mu_{t}(\varphi)$ satisfies the following equation
(i.e., filtering equation)
\begin{equation}\label{general filtering equation}
%\left\{
\begin{aligned}
d\mu_{t}(\varphi)=\mu_{t}(Q\varphi)dt
+[\mu_{t}(h\varphi)-\mu_{t}(\varphi)\mu_{t}(h)]d\nu_{t}.
\end{aligned}
%\right.
\end{equation}
Recall that $\pi_{t}$ is defined by (\ref{indicator function})
in which $\varphi(\alpha_{t})=1_{\{\alpha_{t}=1\}}$.
In this case,
\begin{equation*}
%\left\{
\begin{aligned}
\mu_{t}(Q\varphi)
=&E[Q\varphi(\cdot)(\alpha_{t})|\mathcal{F}_{t}^{X}]\\
=&Q\varphi(\cdot)(1)\pi_{t}+Q\varphi(\cdot)(2)(1-\pi_{t})\\
=&\lambda_{1}[\varphi(2)-\varphi(1)]\pi_{t}
+\lambda_{2}[\varphi(1)-\varphi(2)](1-\pi_{t})\\
=&\lambda_{1}[0-1]\pi_{t}
+\lambda_{2}[1-0](1-\pi_{t})\\
=&-\lambda_{1}\pi_{t}+\lambda_{2}(1-\pi_{t}).
\end{aligned}
%\right.
\end{equation*}
Similarly,
\begin{equation*}
%\left\{
\begin{aligned}
\mu_{t}(h)
=&E[h(t,X_{t},\alpha_{t},v_{t})|\mathcal{F}_{t}^{X}]\\
=&h(t,X_{t},1,v_{t})\pi_{t}+h(t,X_{t},2,v_{t})(1-\pi_{t}),
\end{aligned}
%\right.
\end{equation*}
and
\begin{equation*}
%\left\{
\begin{aligned}
\mu_{t}(h\varphi)
=&E[h(t,X_{t},\alpha_{t},v_{t})1_{\{\alpha_{t}=1\}}|\mathcal{F}_{t}^{X}]\\
=&h(t,X_{t},1,v_{t})\pi_{t}.
\end{aligned}
%\right.
\end{equation*}
Then,
\begin{equation*}
%\left\{
\begin{aligned}
&\mu_{t}(h\varphi)-\mu_{t}(\varphi)\mu_{t}(h)\\
=&h(t,X_{t},1,v_{t})\pi_{t}-[h(t,X_{t},1,v_{t})\pi_{t}+h(t,X_{t},2,v_{t})(1-\pi_{t})]\pi_{t}\\
=&[h(t,X_{t},1,v_{t})-h(t,X_{t},2,v_{t})]\pi_{t}(1-\pi_{t}).
\end{aligned}
%\right.
\end{equation*}
Thus, the general filtering equation (\ref{general filtering equation})
reduces to the second part of (\ref{CO state equation}).

On the other hand,
\begin{equation*}
%\left\{
\begin{aligned}
dX_{t}=&\sigma(t,X_{t},v_{t})dY_{t}\\
=&\sigma(t,X_{t},v_{t})[\mu_{t}(h)dt+d\nu_{t}]\\
=&\sigma(t,X_{t},v_{t})E[h(t,X_{t},\alpha_{t},v_{t})|\mathcal{F}_{t}^{X}]dt+\sigma(t,X_{t},v_{t})d\nu_{t}\\
=&E[b(t,X_{t},\alpha_{t},v_{t})|\mathcal{F}_{t}^{X}]dt+\sigma(t,X_{t},v_{t})d\nu_{t}\\
=&[b(t,X_{t},1,v_{t})\pi_{t}+b(t,X_{t},2,v_{t})(1-\pi_{t})]dt+\sigma(t,X_{t},v_{t})d\nu_{t},
\end{aligned}
%\right.
\end{equation*}
which leads to the first part of (\ref{CO state equation}).

Finally, the cost functional (\ref{PO cost functional}) becomes
\begin{equation*}
%\left\{
\begin{aligned}
J(v)
=&E\bigg[\int_{0}^{T}f(t,X_{t},\alpha_{t},v_{t})dt+g(X_{T},\alpha_{T})\bigg]\\
=&E\bigg[\int_{0}^{T}E[f(t,X_{t},\alpha_{t},v_{t})|\mathcal{F}_{t}^{X}]dt
+E[g(X_{T},\alpha_{T})|\mathcal{F}_{T}^{X}]\bigg]\\
=&E\bigg[\int_{0}^{T}[f(t,X_{t},1,v_{t})\pi_{t}+f(t,X_{t},2,v_{t})(1-\pi_{t})]dt\\
&+g(X_{T},1)\pi_{T}+g(X_{T},2)(1-\pi_{T})\bigg].
\end{aligned}
%\right.
\end{equation*}
The proof is now completed.
\end{proof}
Now, the partially observed control problem is converted to a completely observed one.
In fact, the two-dimensional state equation (\ref{CO state equation}) is driven by $\nu_{t}$,
whose natural filtration is denoted as $\{\mathcal{F}_{t}^{\nu}\}_{t\in [0,T]}$,
and the available information $\mathcal{F}_{t}^{X}$ (or, $\mathcal{F}_{t}^{Y}$)
contains $\mathcal{F}_{t}^{\nu}$ (see Lemma \ref{lemma innovation process}).
Then, our aim is to find a $u\in\mathcal{U}_{ad}$ to minimize (\ref{CO cost functional})
subject to (\ref{CO state equation}).
\begin{remark}
Note that the non-linear filtering method developed above is novel and
essentially different from the traditional approaches such as Wonham \cite{Wonham1965},
Liptser and Shiryayev \cite{LipsterShiryayev1977}, and Bj\"{o}rk \cite{Bjork1980}.
It provides a general filtering equation (\ref{general filtering equation})
for the conditional expectation of a non-linear function of the Markov chain,
which appears for the first time in the literature. In particular, when $\varphi$ takes
an indicator function as (\ref{indicator function}), the general filtering equation
(\ref{general filtering equation}) reduces to the second part of (\ref{CO state equation}),
which coincides with that in \cite{Wonham1965,LipsterShiryayev1977,Bjork1980}.
\end{remark}

\section{Maximum principle}\label{section MP}

For convenience, we rewrite the state equation (\ref{CO state equation})
and cost functional (\ref{CO cost functional}) of the completely observed problem
in a more compact form.

Denote
\begin{equation*}
%\left\{
\begin{aligned}
\Theta_{t}=\left[
   \begin{array}{c}
     X_{t} \\
     \pi_{t} \\
   \end{array}
 \right].
\end{aligned}
%\right.
\end{equation*}
Then the state equation (\ref{CO state equation}) can be rewritten as
\begin{equation}\label{matrix state equation}
\left\{
\begin{aligned}
d\Theta_{t}=&B(t,\Theta_{t},v_{t})dt+\Sigma(t,\Theta_{t},v_{t})d\nu_{t},\\
\Theta_{0}=&\left[
              \begin{array}{c}
                x \\
                \pi \\
              \end{array}
            \right]\in R\times\{1,2\},
\end{aligned}
\right.
\end{equation}
where we denote
\begin{equation*}
%\left\{
\begin{aligned}
B(t,\Theta,v)=&\left[
   \begin{array}{c}
     b(t,x,1,v)\pi+b(t,x,2,v)(1-\pi) \\
     -\lambda_{1}\pi+\lambda_{2}(1-\pi) \\
   \end{array}
 \right],\\
\Sigma(t,\Theta,v)=&\left[
   \begin{array}{c}
     \sigma(t,x,v) \\
     (h(t,x,1,v)-h(t,x,2,v))\pi(1-\pi) \\
   \end{array}
 \right].
\end{aligned}
%\right.
\end{equation*}
The cost functional (\ref{CO cost functional}) is rewritten as
\begin{equation}\label{matrix cost functional}
%\left\{
\begin{aligned}
J(v)=E\bigg[\int_{0}^{T}F(t,\Theta_{t},v_{t})dt+G(\Theta_{T})\bigg],
\end{aligned}
%\right.
\end{equation}
where we denote
\begin{equation*}
%\left\{
\begin{aligned}
F(t,\Theta,v)=&f(t,x,1,v)\pi+f(t,x,2,v)(1-\pi),\\
G(\Theta)=&g(x,1)\pi+g(x,2)(1-\pi).
\end{aligned}
%\right.
\end{equation*}
The completely observed problem is to minimize (\ref{matrix cost functional})
subject to (\ref{matrix state equation}).
It is a variant of the classical optimal control problem with no Markov chain
and the information flow $\mathcal{F}_{t}^{X}$ is slightly bigger than
$\mathcal{F}_{t}^{\nu}$ generated by the driving Brownian motion $\nu_{t}$.
Note that the optimal filter $\pi_{t}$ is a conditional probability, so $\pi_{t}\in [0,1]$.
As a consequence, under Assumptions (A1)-(A5), the coefficients of the
completely observed problem satisfy the usual conditions that are needed
in a maximum principle. Hence, we can adopt the existing results
(see, for example, Peng \cite{Peng1993}) with minor modification
to establish our maximum principle. We list the details in the following.

Let $u_{t}\in\mathcal{U}_{ad}$ be an optimal control. Let $X_{t}$ be the corresponding
optimal trajectory and $\pi_{t}$ be the corresponding optimal filter.
Let $v_{t}\in\mathcal{U}_{ad}$ be such that $u_{t}+v_{t}\in\mathcal{U}_{ad}$.

Denote
\begin{equation*}
%\left\{
\begin{aligned}
\Gamma_{t}=\left[
   \begin{array}{c}
     \xi_{t} \\
     \eta_{t} \\
   \end{array}
 \right].
\end{aligned}
%\right.
\end{equation*}
The variational equation is given by
\begin{equation}\label{CO variational equation}
\left\{
\begin{aligned}
d\Gamma_{t}=&[B_{\Theta}(t,\Theta_{t},u_{t})\Gamma_{t}+B_{v}(t,\Theta_{t},u_{t})v_{t}]dt
+[\Sigma_{\Theta}(t,\Theta_{t},u_{t})+\Sigma_{v}(t,\Theta_{t},u_{t})v_{t}]d\nu_{t},\\
\Gamma_{0}=&\left[
   \begin{array}{c}
     0 \\
     0 \\
   \end{array}
 \right],
\end{aligned}
\right.
\end{equation}
where
\begin{equation*}
%\left\{
\begin{aligned}
B_{\Theta}=\left[
    \begin{array}{cc}
      b_{x}(t,x,1,u)\pi+b_{x}(t,x,2,u)(1-\pi) & b(t,x,1,u)-b(t,x,2,u) \\
      0 & -\lambda_{1}-\lambda_{2} \\
    \end{array}
  \right],
\end{aligned}
%\right.
\end{equation*}
\begin{equation*}
%\left\{
\begin{aligned}
B_{v}=\left[
   \begin{array}{c}
     b_{v}(t,x,1,u)\pi+b_{v}(t,x,2,u)(1-\pi) \\
     0 \\
   \end{array}
 \right],
\end{aligned}
%\right.
\end{equation*}
\begin{equation*}
%\left\{
\begin{aligned}
\Sigma_{\Theta}=\left[
   \begin{array}{cc}
     \sigma_{x}(t,x,u) & 0 \\
     (h_{x}(t,x,1,u)-h_{x}(t,x,2,u))\pi(1-\pi) & (h(t,x,1,u)-h(t,x,2,u))(1-2\pi) \\
   \end{array}
 \right],
\end{aligned}
%\right.
\end{equation*}
\begin{equation*}
%\left\{
\begin{aligned}
\Sigma_{v}=\left[
   \begin{array}{c}
     \sigma_{v}(t,x,u) \\
     (h_{v}(t,x,1,u)-h_{v}(t,x,2,u))\pi(1-\pi) \\
   \end{array}
 \right].
\end{aligned}
%\right.
\end{equation*}
It turns out that the variational equation (\ref{CO variational equation})
admits a unique solution $\Gamma_{t}$ as it is a linear SDE with bounded coefficients.

Then the variational inequality has the following form:
\begin{equation*}
%\left\{
\begin{aligned}
&E\bigg[\int_{0}^{T}[\langle F_{\Theta}(t,\Theta_{t},u_{t}),\Gamma_{t}\rangle
+F_{v}(t,\Theta_{t},u_{t})v(t)]dt
+\langle G_{\Theta}(\Theta_{T}),\Gamma_{T}\rangle\bigg]\geq0,
\end{aligned}
%\right.
\end{equation*}
where
\begin{equation*}
%\left\{
\begin{aligned}
F_{\Theta}(t,\Theta,u)=\left[
            \begin{array}{c}
            f_{x}(t,x,1,u)\pi+f_{x}(t,x,2,u)(1-\pi) \\
            f(t,x,1,u)-f(t,x,2,u) \\
            \end{array}
           \right],
\end{aligned}
%\right.
\end{equation*}
\begin{equation*}
%\left\{
\begin{aligned}
F_{v}(t,\Theta,u)=f_{v}(t,x,1,u)\pi+f_{v}(t,x,2,u)(1-\pi),
\end{aligned}
%\right.
\end{equation*}
\begin{equation*}
%\left\{
\begin{aligned}
G_{\Theta}(\Theta)=\left[
               \begin{array}{c}
                g_{x}(x,1)\pi+g_{x}(x,2)(1-\pi) \\
                g(x,1)-g(x,2) \\
               \end{array}
              \right].
\end{aligned}
%\right.
\end{equation*}
Denote
\begin{equation*}
%\left\{
\begin{aligned}
\Phi_{t}=\left[
   \begin{array}{c}
     p_{t} \\
     k_{t} \\
   \end{array}
 \right],\quad \Lambda_{t}=\left[
   \begin{array}{c}
     P_{t} \\
     K_{t} \\
   \end{array}
 \right].
\end{aligned}
%\right.
\end{equation*}
To derive the maximum principle, we introduce the following adjoint equation:
\begin{equation}\label{CO adjoint equation}
\left\{
\begin{aligned}
d\Phi_{t}
=&-[B_{\Theta}^{\top}(t,\Theta,u)\Phi_{t}+\Sigma_{\Theta}^{\top}(t,\Theta,u)\Lambda_{t}
+F_{\Theta}(t,\Theta,u)]dt
+\Lambda_{t}d\nu_{t},\\
\Phi_{T}=&G_{\Theta}(\Theta_{T}),
\end{aligned}
\right.
\end{equation}
which is a standard and well-posed backward stochastic differential equation (BSDE)
and admits a unique solution $(\Phi_{t},\Lambda_{t})$.

Define the Hamiltonian
$H:[0,T]\times (R\times[0,1])\times U\times R^{2}\times R^{2}\mapsto R$ as
\begin{equation*}
%\left\{
\begin{aligned}
H(t,\Theta,v,\Phi,\Lambda)
=\langle\Phi,B(t,\Theta,v)\rangle+\langle\Lambda,\Sigma(t,\Theta,v)\rangle+F(t,\Theta,v).
\end{aligned}
%\right.
\end{equation*}
In view of Peng \cite[Theorem 4.4]{Peng1993},
we establish the stochastic maximum principle for our hybrid optimal control problem
under partial observation.
\begin{theorem}\label{theorem MP}
Let Assumptions (A1)-(A5) hold. Let $u_{t}\in\mathcal{U}_{ad}$ be an optimal control
and let $\Theta_{t}$ be the corresponding solution to (\ref{matrix state equation}).
Then, we have
\begin{equation*}
%\left\{
\begin{aligned}
H_{v}(t,\Theta_{t},u_{t},\Phi_{t},\Lambda_{t})=0,
\end{aligned}
%\right.
\end{equation*}
where $(\Phi_{t},\Lambda_{t})$ is the unique solution to the
adjoint equation (\ref{CO adjoint equation}).
\end{theorem}

\section{LQ problem}\label{section LQ}

Theoretically, the maximum principle established in the previous section
characterizes the optimal controls through some necessary conditions.
In this section, we provide an LQ control problem and show how to solve
the problem using our maximum principle.

The LQ hybrid optimal control problem under partial observation is
\begin{equation*}
%\left\{
\begin{aligned}
J(u)=\min_{v\in\mathcal{U}_{ad}}J(v),
\end{aligned}
%\right.
\end{equation*}
where
\begin{equation*}
%\left\{
\begin{aligned}
J(v)=\frac{1}{2}E\bigg[\int_{0}^{T}
[Q(\alpha_{t})X_{t}^{2}+R(\alpha_{t})v_{t}^{2}]dt
+G(\alpha_{T})X_{T}^{2}\bigg],
\end{aligned}
%\right.
\end{equation*}
subject to
\begin{equation*}
\left\{
\begin{aligned}
dX_{t}=&[a(\alpha_{t})X_{t}+b(\alpha_{t})v_{t}]dt+\sigma dW_{t},\\
X_{0}=&x\in R,\quad \alpha_{0}=i\in\mathcal{M}.
\end{aligned}
\right.
\end{equation*}
Here, $a(i)$, $b(i)$, $\sigma$, $Q(i)$, $R(i)$, $G(i)$, $i\in\mathcal{M}$, are constants.
Note that the above LQ problem has been studied by
Zhang and Yin \cite{ZhangYin1999TAC} under a completely observed setup.

In the section, we assume:

\textbf{(A6)} $\sigma>0$ and $R(i)>0$, $i\in\mathcal{M}$.

The two-dimensional state process $\Theta_{t}=(X_{t},\pi_{t})^{\top}$
satisfies the following equation:
\begin{equation*}
\left\{
\begin{aligned}
d\Theta_{t}=&B(t,\Theta_{t},v_{t})dt+\Sigma(t,\Theta_{t},v_{t})d\nu_{t},\\
\Theta_{0}=&\left[
              \begin{array}{c}
                x \\
                \pi \\
              \end{array}
            \right]\in R\times \{0,1\},
\end{aligned}
\right.
\end{equation*}
where
\begin{equation*}
%\left\{
\begin{aligned}
B(t,\Theta,v)=&\left[
   \begin{array}{c}
     (a(1)x+b(1)v)\pi+(a(2)x+b(2)v)(1-\pi) \\
     -\lambda_{1}\pi+\lambda_{2}(1-\pi) \\
   \end{array}
 \right]\doteq\left[
                \begin{array}{c}
                  B_{1}(t,\Theta,v) \\
                  B_{2}(t,\Theta,v) \\
                \end{array}
              \right],\\
\Sigma(t,\Theta,v)=&\left[
   \begin{array}{c}
     \sigma \\
     \sigma^{-1}((a(1)-a(2))x+(b(1)-b(2))v)\pi(1-\pi) \\
   \end{array}
 \right]\doteq\left[
                \begin{array}{c}
                  \Sigma_{1}(t,\Theta,v) \\
                  \Sigma_{2}(t,\Theta,v) \\
                \end{array}
              \right].
\end{aligned}
%\right.
\end{equation*}
The cost functional becomes
\begin{equation*}
%\left\{
\begin{aligned}
J(v)=E\bigg[\int_{0}^{T}F(t,\Theta_{t},v_{t})dt
+G(\Theta_{T})\bigg],
\end{aligned}
%\right.
\end{equation*}
where
\begin{equation*}
%\left\{
\begin{aligned}
F(t,\Theta,v)=\frac{1}{2}[(Q(1)x^{2}+R(1)v^{2})\pi
+(Q(2)x^{2}+R(2)v^{2})(1-\pi)],
\end{aligned}
%\right.
\end{equation*}
\begin{equation*}
%\left\{
\begin{aligned}
G(\Theta)=\frac{1}{2}[G(1)x^{2}\pi+G(2)x^{2}(1-\pi)].
\end{aligned}
%\right.
\end{equation*}
Let $u_{t}\in\mathcal{U}_{ad}$ be an optimal control.
In this case, the adjoint equation reads
\begin{equation}\label{LQ adjoint equation}
\left\{
\begin{aligned}
d\Phi_{t}
=&-[B_{\Theta}^{\top}(t,\Theta_{t},u_{t})\Phi_{t}
+\Sigma_{\Theta}^{\top}(t,\Theta_{t},u_{t})\Lambda_{t}
+F_{\Theta}(t,\Theta_{t},u_{t})]dt
+\Lambda_{t}d\nu_{t},\\
\Phi_{T}=&G_{\Theta}(\Theta_{T}),
\end{aligned}
\right.
\end{equation}
where
\begin{equation*}
%\left\{
\begin{aligned}
B_{\Theta}=\left[
             \begin{array}{cc}
               a(1)\pi+a(2)(1-\pi) & (a(1)-a(2))x+(b(1)-b(2))u \\
               0 & -\lambda_{1}-\lambda_{2} \\
             \end{array}
           \right],
\end{aligned}
%\right.
\end{equation*}
\begin{equation*}
%\left\{
\begin{aligned}
\Sigma_{\Theta}=\left[
             \begin{array}{cc}
               0 & 0 \\
               \sigma^{-1}(a(1)-a(2))\pi(1-\pi) & \sigma^{-1}((a(1)-a(2))x+(b(1)-b(2))u)(1-2\pi) \\
             \end{array}
           \right],
\end{aligned}
%\right.
\end{equation*}
\begin{equation*}
%\left\{
\begin{aligned}
F_{\Theta}=\left[
             \begin{array}{c}
               Q(1)x\pi+Q(2)x(1-\pi) \\
               \frac{1}{2}[(Q(1)-Q(2))x^{2}+(R(1)-R(2))u^{2}] \\
             \end{array}
           \right],
\end{aligned}
%\right.
\end{equation*}
\begin{equation*}
%\left\{
\begin{aligned}
G_{\Theta}=\left[
             \begin{array}{c}
               G(1)x\pi+G(2)x(1-\pi) \\
               \frac{1}{2}(G(1)-G(2))x^{2} \\
             \end{array}
           \right].
\end{aligned}
%\right.
\end{equation*}
The corresponding Hamiltonian is given by
\begin{equation*}
%\left\{
\begin{aligned}
&H(t,\Theta,v,\Phi,\Lambda)\\
=&\bigg\langle\left[
              \begin{array}{c}
                p \\
                k \\
              \end{array}
            \right],
            \left[
              \begin{array}{c}
                B_{1}(t,\Theta,v) \\
                B_{2}(t,\Theta,v) \\
              \end{array}
            \right]\bigg\rangle
+\bigg\langle\left[
              \begin{array}{c}
                P \\
                K \\
              \end{array}
            \right],
            \left[
              \begin{array}{c}
                \Sigma_{1}(t,\Theta,v) \\
                \Sigma_{2}(t,\Theta,v) \\
              \end{array}
            \right]\bigg\rangle+F(t,\Theta,v).
\end{aligned}
%\right.
\end{equation*}
From the maximum principle (Theorem \ref{theorem MP}),
it is necessary for the optimal control $u_{t}$ to satisfy
\begin{equation*}
%\left\{
\begin{aligned}
H_{v}(t,\Theta_{t},u_{t},\Phi_{t},\Lambda_{t})=0,
\end{aligned}
%\right.
\end{equation*}
i.e.,
\begin{equation*}
%\left\{
\begin{aligned}
&p_{t}(b(1)\pi_{t}+b(2)(1-\pi_{t}))
+K_{t}\sigma^{-1}(b(1)-b(2))\pi_{t}(1-\pi_{t})\\
&+(R(1)\pi_{t}+R(2)(1-\pi_{t}))u_{t}=0.
\end{aligned}
%\right.
\end{equation*}
Then we have
\begin{equation}\label{LQ optimal control}
%\left\{
\begin{aligned}
u_{t}=&-(R(1)\pi_{t}+R(2)(1-\pi_{t}))^{-1}\\
&\times[p_{t}(b(1)\pi_{t}+b(2)(1-\pi_{t}))
+K_{t}\sigma^{-1}(b(1)-b(2))\pi_{t}(1-\pi_{t})],
\end{aligned}
%\right.
\end{equation}
which is observable and takes a non-linear form in terms of
the optimal filter and adjoint process.
\begin{proposition}
Let Assumption (A6) hold. Then, an observable optimal control for the LQ problem
under partial observation is given by (\ref{LQ optimal control}), where $p_{t}$
and $K_{t}$ are part of the unique solution to the adjoint equation (\ref{LQ adjoint equation}).
\end{proposition}

\section{Concluding remarks}\label{section concluding remarks}

In this paper, a novel non-linear filtering method for hybrid diffusions is developed,
which gives a filtering equation (\ref{general filtering equation})
satisfied by the general filtering process.
A new maximum principle is established and turns out to be a direct and convenient
way to study the hybrid optimal control problem under partial observation,
in which the Markov chain cannot be observed. As an application, the maximum principle
is applied to solve an LQ problem to get some observable optimal controls.

This paper, we believe, has posed more questions than answers.
The problem formulation considered in this paper is a simple but illustrative one,
extensions to more general problems may open up a new avenue
for stochastic filtering and stochastic control theory.
On the other hand, the maximum principle derived should have
a wide range of applications in various fields,
such as finance, management, engineering, and so on.
These topics in practice will be considered in our future study.

\end{document}